\documentclass[12pt]{article}

\usepackage[T1]{fontenc} 

\usepackage[utf8]{inputenc} 

\usepackage{amsmath,amssymb,amsthm} 

\usepackage{varioref} 

\usepackage{datetime}

\usepackage{pifont}

\newcommand{\li}{\ell_\infty}

\renewcommand{\sectionmark}[1]{\markright{{\jobname{\textrm.tex}\ \ \ \ \ \ \ \ \ \today\ \ \ \ }}} 

\usepackage{epsfig}
\usepackage{amssymb}
\usepackage{graphicx}

\usepackage[usenames,dvipsnames]{color}

\usepackage{setspace}
\usepackage{enumitem}
\setlist{nolistsep}

\usepackage[mathscr]{eucal}

\usepackage{pifont}

\newcommand{\F}{{\mbox F}}

\newcommand{\PG}{{\textup{PG}}}

\newcommand{\PGammaL}{\textup{P}\Gamma{\textup {L}}}
\newcommand{\Fqq}{\mathbb{F}_{q^2}}

\newcommand{\Fq}{\mathbb{F}_{q}}

\newcommand{\T}{\mathcal{T}}

\newcommand\Tr{\mbox{Tr}}

\newcommand\U{{\mathcal U}}
\newcommand\Uab{{\mathcal U}_{a,b}}
\newcommand\Uabp{{\mathcal U}_{a',b'}}

\newtheorem{theorem}{Theorem}

\newtheorem{lemma}[theorem]{Lemma}
\newtheorem{corollary}[theorem]{Corollary}
\newtheorem{conjecture}[theorem]{{\color{CornflowerBlue}{Conjecture}}}
\newtheorem{defn}[theorem]{Definition}
\theoremstyle{remark} 

\newtheorem{example}{Example}

\newcommand{\Label}{\label}
\newcommand{\Labele}{\label}

\usepackage{pdfpages}

\begin{document}

\title{A Note on Fano Planes in Orthogonal Buekenhout-Metz Unitals of Even Order}
\author{Wen-Ai Jackson\footnote{School of Mathematical Sciences, University of Adelaide, Adelaide 5005, Australia, Email:\tt{wen.jackson@adelaide.edu.au}}\ \  and Peter Wild\footnote{Royal Holloway, University of London, Egham TW20 0EX, Email:\tt{peterrwild@gmail.com}}}


\maketitle

\begin{abstract}An O'Nan configuration in a unital is a set of four lines forming a quadrilateral, where the six intersections of pairs of lines are points of the unital.  In 2019 Feng and Li~\cite{FengLi2019} elegantly construct O'Nan configurations in orthogonal and Tits Buekenhout-Metz unitals. We extend their work by extending their construction to a Fano plane embedded in the orthogonal Buekenhout-Metz unital of even order. {We deduce that there exist O'Nan configurations in orthogonal Buekenhout-Metz unitals different to those of Feng and Li, and make a conjecture about Fano planes embedded in orthogonal Buekenhout-Metz unitals.}
\end{abstract}

\section{Introduction}\Label{sec:intro}

A \emph{unital of order $n$} is a 2--$(n^3+1,n+1,1)$ design. Unitals of prime power order $q$ exist in the Desarguesian projective planes $\PG(2,q^2)$, as the absolute points and lines of a hermitian polarity. Such unitals are called \emph{classical} unitals. Buekenhout \cite{Buekenhout1976} constructed unitals in 2--dimensional translation planes. This was generalised by Metz \cite{Metz1979} to construct non-classical unitals in $\PG(2,q^2)$. These unitals are called Buekenhout-Metz (BM) unitals. They have the property that every line of $\PG(2,q)$ meets the unital in 1 or $q+1$ points. Further, the unital has exactly one point on the line at infinity, called the \emph{special} point. For a comprehensive treatise on unitals, see Barwick and Ebert \cite{thebook}.

An \emph{O'Nan configuration} in a unital is a set of four distinct lines forming a quadrilateral, whose six points of intersection lie on the unital.  In 1972, O'Nan \cite{ONan1972} noted that these configurations do not occur in the classical unital.  It was conjectured by Piper \cite{Piper1979} that this property characterizes the classical unital within the class of all unitals.

In 2019, Feng and Li \cite{FengLi2019} showed the existence of O'Nan configurations in orthogonal BM unitals in $\PG(2,q^2)$, for both $q$ even and odd.  They also show the existence of O'Nan configurations in the Buekenhout-Tits unital.

In this paper, we use their results to show the existence of Fano planes in BM unitals of even order, which supports the following conjecture:

\begin{conjecture}
  Suppose $q$ is even. Every non-classical unital $\U$ in $\PG(2,q^2)$ contains a Fano plane  if $q\ge 4$.
\end{conjecture}

\section{Buekenhout-Metz unitals}

\subsection{Field properties for even $q$}\Label{sec:eqns}

We give properties of $\Fqq$ where $q$ is an even prime power.

In general, consider a prime $p$ and a positive integer $m$ with  $q=p^m>2$. Suppose that $d$ is any divisor of $m$, then the trace function from $\Fq$ to $\mathbb F_{p^d}$ is
\[
\Tr_{p^m/p^d}(x)=x+x^{p^d}+x^{p^{d+1}}+\cdots+x^{p^{m-d}}.
\]

Suppose now that $q=2^m>2 $. The \emph{absolute trace} of an element $x$ of $\Fq$ is $\Tr_{2^m/2}(x)$.

Let $\T_0,\T_1$ be the set of elements of $\Fq$ with absolute trace 0 and 1 respectively. From~\cite[Lemma 4.21]{thebook}, we can choose $\delta\in\Fqq\backslash\Fq$ such that $\delta^q=\delta+1$, $\delta^2=\delta+v$ for some $v\in\T_1,v\ne 1$. Then $\{1,\delta\}$ is a basis for $\Fqq$ over $\Fq$.

For any element $x$ of $\Fqq$ we will also use the notation
\[
\boxed{x}=x+x^q=\Tr_{q^2/q}(x).
\]

For $k=k_0+k_1\delta$,  with $k_0,k_1\in\Fq$ it is easy to calculate $k^q=(k_0+k_1)+k_1\delta = k+k_1$, $k^{q+1}=k^qk=k_0^2+k_0k_1+vk_1^2$, $\boxed{k}= k+k^q=k_1$, $\boxed{k\delta}=(k\delta)+(k\delta)^q=k_0+k_1$ and $\boxed{k\delta^q}=(k\delta^q)+(k^q\delta)=k_0$.

\subsection{Coordinates of the orthogonal BM unitals}
Ebert \cite{ebertBMeven} showed that an orthogonal Buekenhout-Metz unital $\Uab$ in $\PG(2,q^2)$ is projectively equivalent to the following set of points
\[
\Uab=\{(x,ax^2+bx^{q+1}+r,1)\bigm| r\in\Fq,\ x\in\Fqq\}\cup\{T_\infty=(0,1,0)\}.
\]
where the discriminant $d=\frac{a^{q+1}}{(b^q+b)^2}$ belongs to $\T_0$, the set of elements of \emph{absolute trace} 0 of $\Fq$.

Further $\Uab$ is classical if and only if $a=0$.

Consequently, an affine point $P=(x,y,1)$ belongs to $\Uab$ if and only if
\begin{equation}
 ax^2+a^qx^{2q}+(b+b^q)x^{q+1}+y+y^q =0.\Labele{eqn:upt}
\end{equation}


\subsection{The automorphism group of the orthogonal BM unital, even $q$}\Label{sec:autosG}

This was investigated by Ebert \cite{ebertBMeven} but we use the notation from \cite[Theorem 4.23]{thebook}.

We assume that the BM unital $\Uab$ is not classical, that is, $a\ne 0$. Let $\F$ be the smallest subfield of $\Fq$ containing the discriminant $d$. Let $G$ be the subgroup of the automorphism group $\PGammaL(3,q^2)$ of $\PG(2,q^2)$ leaving $\Uab$ invariant. Then $G$ also fixes $T_\infty$.

The order of $G$ is $mq^3(q-1)$ where $m$ is the dimension of $\Fqq$ over $\F$. $G$ has four point orbits, namely, $\{T_\infty\}$, $\Uab\backslash T_\infty$, $\li\backslash T_\infty$ and $\PG(2,q^2)\backslash(\Uab\cup\li)$.

\subsection{The equivalence of orthogonal BM unitals, even $q$}\Label{sec:equiv}
From Ebert~\cite{ebertBMeven}, two BM unitals $\Uab$, $\U_{a',b'}$ are equivalent if and only if we can find $v\in\Fq^*$, $\gamma\in\Fqq^*$ and $u\in\Fq$ and $\tau$ a field automorphism of $\Fq$ with
\[
(a',b')=(a^\tau\gamma^2 v,\ \ b^\tau\gamma^{q+1}v+u).
\]

For $q=2^n$, the number of inequivalent orthogonal BM unitals in $\PG(2,q^2)$, for $n=2,3,4,5,6$ is
$2,2,4,4,8$
respectively, one of which is the classical unital \cite[Theorem 4.14]{thebook}.

\begin{lemma} \Label{lem:ab}For even $q$, if\ \ $\Uab $ is any non-classical orthogonal BM unital, then without loss of generality we may assume
  \begin{enumerate}
  \item $a=1$, or
  \item $a\in\Fq^*$ and $b=\delta$.
  \end{enumerate}
\end{lemma}
\begin{proof}

  From above, given $\Uab$, then $\U_{a',b'}$ with
  \[ (a',b')=(a^\tau\gamma^2\nu,\ \ b^\tau\gamma^{q+1}+u)
  \]
  is an equivalent unital if $\nu\in\Fq^*$, $\gamma\in\Fqq^*$ and $u\in\Fq$.
For Part 1, we take $\tau=$ identity, $\gamma=1/\sqrt a$, $\nu=1$ and $u=0$ to get $(1,b/(\sqrt a)^{q+1})$, as required.

For Part 2, the proof is in Feng and Li~\cite[Section 3.2]{FengLi2019}.
\end{proof}
\subsection{Fano planes in Buekenhout-Metz unitals}
It is well known that no O'Nan configuration contains the special (infinite) point $T_\infty=(0,1,0)$ of a Buekenhout-Metz unital \cite[Lemma 7.42]{thebook}. Hence it follows there is no Fano plane containing the special point $T_\infty$. However, it is possible that a line of the Fano plane, regarded as a line of $\PG(2,q^2)$,  contains $T_\infty$.

\begin{defn} In $\PG(2,q)$, $q$ even, consider a BM unital  $\U$ with special point $T_\infty$ at infinity.
  We say a Fano plane in a BM unital embedded in $\PG(2,q^2)$ with special point $T_\infty$ is a \emph{BM-special Fano plane} if one of the lines of the Fano plane, when considered as a line as of $\PG(2,q^2)$, contains the special point $T_\infty$;
  Otherwise, we call it a  \emph{BM-ordinary Fano plane}.
\end{defn}

\subsection{Contribution of  this paper}

Feng and Li \cite{FengLi2019} have three  beautiful constructions of O'Nan configurations. In this paper we consider  their  O'Nan configurations in the orthogonal case for even $q$. Their O'Nan configurations have a line which contains the special point $T_\infty$.

We begin by showing that the O'Nan configuration of Feng and Li can be extended to a Fano plane.

\begin{theorem}\Label{thm:qeven}
  The O'Nan configuration of Feng and Li  \cite{FengLi2019} for orthogonal BM unitals of even order, can be extended to a Fano plane contained within the BM unital, that is, to a BM-special Fano plane.
\end{theorem}

We further investigate the BM-special Fano planes, and prove the following:

\begin{theorem}\Label{thm:numeven}
  Let $q=2^m>2$. Let $\U_1(1,b)$ and $\U_2(1,b')$ be any two non-classical orthogonal BM unitals of $\PG(2,q^2)$. Then the number of BM-special Fano planes in $\U_1$ and $\U_2$ are the same.
\end{theorem}

Note that  for $q=4,8$ there is only one equivalence class of non-classical unitals, so the above theorem holds trivially.

Since a Fano plane contains seven O'Nan configurations we have the following Corollary.

\begin{corollary}
  Let $q=2^m>2$. In orthogonal BM unitals of $\PG(2,q^2)$, there exist O'Nan configurations, where none of the four lines of the configuration is incident with the special point $T_\infty$ of the BM unital.
\end{corollary}

\section{Main results for BM-unitals, even $q$}

\subsection{Existence of Fano planes in even order orthogonal BM unitals}

We now complete the proof of Theorem~\ref{thm:qeven}.

Feng and Li choose $a\in\Fq^*$ and $\beta=\delta$ (see Lemma~\ref{lem:ab}(2)), and define
\[
\phi\colon \PG(2,q^2)\rightarrow \PG(2,q^2),\quad\phi\colon (x,y,z)\mapsto (x^q,y^q,z^q)
\]
which stabilises $\Uab$, fixing the line $B_\infty=\{(0,r,1)\colon r\in\Fq\}\cup\{(0,1,0)\}$ of $\Uab$, and no other point of $\Uab$.  They show there exists points $Q,R\in\Uab$ with $Q,R\in[\delta,1,0]$, and that if $P=(0,0,1)$, then the lines
\[
QR^\phi,RQ^\phi,QQ^\phi,RR^\phi
\]
are four lines with the six points of intersection $P,Q,R,Q^\phi,R^\phi$ and $M=QQ^\phi\cap RR^\phi\in B_\infty$ lying on $\Uab$, hence forming an O'Nan configuration.

We now show that this completes to a Fano plane.

Let
\[
N=QR\cap Q^\phi R^\phi,
\]
so $N^\phi=Q^\phi R^\phi\cap QR=N$, and $N$ is a fixed point. As $q$ is even, note that $P,M,N$ are collinear, and as $N$ is on $PM$ and $N$ is fixed, and as $B_\infty$ contains the fixed points on $[1,0,0]$ it follows that $N=(0,s,1)$ for some $s\in\Fq$.  So $N\in\Uab$. Hence, the six lines
\[
QR^\phi,RQ^\phi,QQ^\phi,RR^\phi,QR,Q^\phi R^\phi
\]
and the seven points
\[
P,Q,R,Q^\phi,R^\phi,M,N
\]
all belong to $\Uab$, and so form a Fano plane, completing the proof of Theorem~\ref{thm:qeven}.

\subsection{Number of BM-special Fano planes}
We will show that the number of BM-special Fano planes in orthogonal BM-unitals of even order depends only the order of the unital.

We now calculate the general form of a BM-special Fano plane. By Section~\ref{sec:autosG}, we may assume that the point $V=(0,0,1)$ belongs to the Fano plane. As the Fano plane is BM-special, there are a further two points on $VT_\infty$, without loss of generality let these be   $X=(0,s,1)$ and $Y=(0,t,1)$ for $s,t\in\Fq^*$, $s\ne t$. Through each point of the Fano plane lies three lines, so let $\ell_x=[1,x,0]$ be another line through $V$, so $x\in\Fqq^*$, and let the two points of the Fano plane on $\ell_x$ be $P=(xj,j,1)$ and $Q=(xk,k,1)$ for some $j,k\in\Fqq^*$ with $j\ne k$. So $\ell_x=PQV$. The remaining two points $M,N$ of the Fano plane are determine by $M=PX\cap QY$ and $N=PY\cap QX$, and $M,N,V$ are collinear.

Instead of writing $j$, we can define $h\in\Fqq^*$, $h\ne 1$ by $j=kh$. Then every BM-special Fano plane can be described by the parameters $(a,b,x,k,h,s,t)$. However this description is not unique, as we may swap $X$ and $Y$, $P$ and $Q$ and the line $VPQ$ with the line $VMN$. In any case, {the number of such descriptions} will be same for each Fano plane.

\begin{proof} (of Theorem~\ref{thm:numeven})
  Suppose that $a=a'=1$ and put $b=b_0+b_1e$ and $b'=b_0'+b_1'e$. Suppose there exists a BM-special Fano plane in $\Uab$ with parameters $(1,b,x,k,h,s,t)$. Then we will show that  there exists a BM-special Fano plane in $\Uabp$ with parameters $(1,b',x',k',h,s,t)$ with $xk=x'k'$.

  Write $\theta=xk$. We have, with   $h=j/k \in\Fqq$, $P=(\theta h,k h,1)$ and $Q=(\theta,k,1)$ with $\theta,k,h\in\Fqq^*$, $X=(0,s,1)$ and $Y=(0,t,1)$ with $s,t\in\Fq^*$. Let $V=(0,0,1)$. Let $M=PX\cap QY$ and $N=PY\cap QX$, so
  \[
  M=(xkW,kW+U,1)=(\theta W,kW+U), \quad  N=(xkV,kV+Z,1)=(\theta V,kV+Z,1)
  \]
  where
  \[
W=\frac{h(s+t)}{s+h t},\quad  U=\frac{st(1+h)}{s+h t},\quad
V=\frac{h(s+t)}{t+h s},\quad  Z=\frac{st(1+h)}{t+h s}.
\]
The conditions for a complete O'Nan configuration in unital $\U_{a,b}$ with $(a,b)=(1,b_0+b_1\delta)$ are (using the notation as described in Section~\ref{sec:eqns}).
\begin{eqnarray}
  \boxed{(\theta h)^2}+b_1(\theta h)^{q+1}&=&\boxed{kh}\label{eqn:Pd}\\
  \boxed{\theta^2}+b_1\theta^{q+1}&=&\boxed{k}\label{eqn:Qd}\\
  \boxed{(\theta W)^2}+b_1(\theta W)^{q+1}&=&\boxed{kW}+\boxed{U}\label{eqn:Md}\\
  \boxed{(\theta V)^2}+b_1(\theta V)^{q+1}&=&\boxed{kV}+\boxed{Z}\label{eqn:Nd}
\end{eqnarray}

Note if we fix $\theta,h\in\Fqq$, then the first two equations  determines $k=k_0+k_1\delta$, and $k_0$ and $k_1$ depend linearly on $b_1$. In this case, the LHS of (\ref{eqn:Pd}) and (\ref{eqn:Qd}) are fixed linear expressions in $b_1$ and,  as $k+k^q=k_1$, (\ref{eqn:Qd}) determines $k_1$ and using (\ref{eqn:Pd}), since  $\boxed{kh}$ is the $\Fq$-coefficient of $\delta$ of the expression for $kh$, {\it i.e.} $k_0h_1+k_1h_0+k_1h_1$, this determines $k_0$ and hence $k$.

Now suppose the above configuration, $(1,b,x,k,h,s,t)$, is a complete BM-special Fano plane, and so the above equations hold. We will define a configuration in $\U_{a',b'}$ with parameters $(1,b',x',k',h,s,t)$ with $xk=x'k'$, and show that it is a complete BM-special Fano plane in $\U_{a',b'}$. The  configuration $\U_{a',b'}$ has the points $X,Y,V$ in common with the configuration in $\U_{a,b}$ and
 \[
P'=(x'k'h,k'h,1)=(\theta h,k'h,1) \quad Q'=(x'k',k',1)=(\theta,k',1)
\]
and
\[
  M'=(\theta W,k'W+U,1),\quad N'=(\theta V,k'V+Z,1).
  \]
(Note that $W,U,V,Z$ only involve $h,s,t$ which are fixed for the two scenarios $\U_{a,b},\U_{a',b'}$.)

We require that the following equations hold for the configuration to be a Fano plane in $\U_{a',b'}$:
\begin{eqnarray}
  \boxed{(\theta h)^2}+b_1'(\theta h)^{q+1}&=&\boxed{k'h}\label{eqn:Pdd}\\
  \boxed{\theta^2}+b_1'\theta^{q+1}&=&\boxed{k'}\label{eqn:Qdd}\\
  \boxed{(\theta W)^2}+b_1'(\theta W)^{q+1}&=&\boxed{k'W} +\boxed U\label{eqn:Mdd}\\
  \boxed{(\theta V)^2}+b_1'(\theta V)^{q+1}&=&\boxed{k'V}+\boxed Z.\label{eqn:Ndd}
\end{eqnarray}
As before (\ref{eqn:Pdd}) and (\ref{eqn:Qdd}) determine $k'$, and we define $x'=\theta/k'=xk/k'$, so that (\ref{eqn:Pdd}) and (\ref{eqn:Qdd}) hold and express that points $P',Q'$ of the configuration lie in $\U_{a',b'}$. Hence we need to show that (\ref{eqn:Mdd}) and (\ref{eqn:Ndd}) hold.

  We start with (\ref{eqn:Mdd}). As (\ref{eqn:Md}) holds, it is sufficient to show that  (\ref{eqn:Md}) $+$ (\ref{eqn:Mdd}) holds, that is
  \begin{equation}
    (b_1+b_1')\theta^{q+1}W^{q+1}=(k+k')W+(k^q+(k')^q)W^q.\Labele{eqn:join}
    \end{equation}
Looking at RHS (\ref{eqn:join}) we have, using both (\ref{eqn:Pd}), (\ref{eqn:Qd}) and (\ref{eqn:Pdd}), (\ref{eqn:Qdd}),
  \begin{eqnarray*}
    &&(k+k')W+(k^q+(k')^q)W^q\\
    &=&(k+k')\frac{h(s+t)}{s+ht}+(k^q+(k')^q)\frac{h^q(s+t)}{(s+ht)^q}\\
    &=&\frac{s+t}{(s+ht)^{q+1}}\left((k+k')(s+h^qt)h+(k+k')^q(s+ht)h^q\right)\\
    &=&\frac{s+t}{(s+ht)^{q+1}}\left(s[(k+k')h+(k+k')^qh^q]+t[(k+k')h^{q+1}+(k+k')^qh^{q+1}]\right)\\
     &=&\frac{s+t}{(s+ht)^{q+1}}\left(s[\,\boxed{kh}+\boxed{k'h'}]+th^{q+1}[\boxed k+\boxed{k'}]\right)\\
   &=&\frac{s+t}{(s+ht)^{q+1}}(s[\,\boxed{(\theta h)^2}+b_1\theta^{q+1}h^{q+1}+\boxed{(\theta h)^2}+b_1'\theta^{q+1}h^{q+1}]\\
    &&\quad\quad\quad+\ th^{q+1}[\,\boxed{\theta^2}+b_1\theta^{q+1}+\boxed{\theta^2}+b_1'\theta^{q+1}])\\
      &=&\frac{s+t}{(s+ht)^{q+1}}(s(b_1+b_1')\theta^{q+1}h^{q+1}+th^{q+1}(b_1+b_1')\theta^{q+1})\\
      &=&\frac{s+t}{(s+ht)^{q+1}}(s+t)(b_1+b_1')\theta^{q+1}h^{q+1}\\
      &=&\left(\frac{h(s+t)}{(s+ht)}\right)^{q+1}(b_1+b_1')\theta^{q+1}\\
    &=&W^{q+1}(b_1+b_1')\theta^{q+1}
  \end{eqnarray*}
  which is equal to LHS (\ref{eqn:join}), as required, showing that $M'\in \U_{a',b'}$. The argument for $N'\in \U_{a',b'}$ is similar. Thus we have shown that $X,Y,V,P',Q',M',N'$ form a BM-special Fano plane.
    Thus there is a correspondence between BM-special Fano planes in the two unitals and so the two unitals contain the same number of BM-special Fano planes.
\end{proof}

\section{Tits Unitals}

\subsection{Discussion}
Feng and Li~\cite{FengLi2019} also show the existence of O'Nan configurations in Tits unitals. We have investigated their constructions and in some, but not all, cases, their construction can be completed to a Fano plane contained within the Tits unital. In particular, for $q<2048$, our Magma~\cite{magma} calculations report that for each appropriate value of $q$ there is at least one such Fano plane, with the exception of $q=128$.

\subsection{Examples}
The Tits unital has infinite point $(0,1,0)$ and finite points, for $x=x_0+x_1\delta$ and $x_0,x_1,r\in\Fq$
\[
P_{x,r}=(x_0+x_1\delta,\ r+ f(x_0,x_1)\delta, 1)
\]
where $f(x,y)=x^{\tau+1}+xy+y^\tau$.

As $q=2^{2m+1}$ is an odd power of 2, following Feng and Li~\cite{FengLi2019} we suppose $\delta\in\F_4$ satisfies $\delta^2=\delta+1$, further $\delta^q=\delta+1$. Let
\[
\tau\colon \Fqq\rightarrow\Fqq,\quad \tau\colon x\mapsto x^{2^{m+1}},
\]
so if $x\in\Fq$ then $x^{\tau^2}=x^2$.

Here we give an example of a Fano plane contained in the Tits unital obtained by extending the Feng and Li construction.

\begin{example}
  Suppose $q=8$, so $q=2^{2m+1}$ for $m=1$, and $\tau\colon x\mapsto x^2$ for $x\in\Fqq$. Suppose $w$ is a generator of $\Fq^*$, so $w^7=1$ and $w^3=w+1$. Choose $r_1=w^2$, $r_1=w^5$, so we have $V=(0,0,1)$
  \[
R_1=(  0, w^2,   1),\ \ R_2=(  0, w^5 ,  1),\ \ P=(      1 ,  e + w    ,   1)
\]
and
\[
P_1=(w^5+w^6\delta  ,   w^6+\delta  ,         1),\ \ P_2=(w^3+w^6\delta  ,   w^4+\delta ,          1),\ \  M=(w^3 +w^5\delta  ,    1+\delta      ,     1)
\]
We can check that $P_1\in\U_T$, by calculating
\[
f(w^5,w^6)=(w^5)^{4+2}+w^5w^6+(w^6)^4=w^{30}+w^{11}+w^{24}=w^2+w^4+w^3=1
\]
using $w^3=w+1$ so $w^4=w^2+w$. Looking at the $\delta$ part of the second coordinate $w^6+\delta$, this is 1, so this confirms that $P_1\in\U_T$. Similarly for $P_2$. We now check $M$
\[
f(w^3,w^5)=w^{18}+w^8+w^{20}=1
\]
as required. This shows that the points all lie on $U_T$, and it is easy to check they form a Fano plane.
\end{example}

\begin{example}
Suppose now that $q=32$. We choose a generator $w$ of $\F_{32}$ satisfying $w^5=w^2+1$. Her we have $r_1=w^{4}$ and $r_2=w^{5}$. We have $\tau\colon x\mapsto x^{2^3}=x^8$.
 \[
R_1=(  0, w^{4},   1),\ \ R_2=(  0, w^{5} ,  1),\ \ P=(      1 ,  w^{26}+e    ,   1)
\]
and
\[
P_1=(w^{11}+w^{24}\delta  ,   w^{16}+w^{29}\delta  ,         1),\ \ P_2=(w^{15}+w^{8}\delta  ,   w^{20}+w^{13}\delta ,          1),\ \  M=(w^{15} +\delta  ,    w^{4}+w^{29}\delta      ,     1)
\]

\end{example}
Computation shows that no O'Nan configuration contained in the Tits unital constructed by Feng and Li for the case $q=128$ extends to a Fano plane.
\section{Conclusion}
Firstly, consider the non-classical orthogonal BM unitals of even order.
The O'Nans of Feng and Li have a line which contains the special point of the unital. We have shown their construction easily extends to a Fano plane contained in the unital, and consequently, contain an O'Nan configuration whose lines do not contain the special point. Further, we show that the number of these BM-special Fano planes in an orthogonal BM-unital depends only on the order of the unital, that is, it is the same for inequivalent orthogonal BM-unitals of the same even order.

Note that by our calculations using Magma~\cite{magma} we can find BM-ordinary Fano planes for small values of $q$, however is is an open conjecture whether these exist for all even $q$.

Now consider the Tits unital. In this case computations have shown that the O'Nan configurations of Feng and Li  complete to a Fano plane for $q=8,32,512$, and does not for $q=128$. Separate computations have shown however, that Fano planes exists in the Tits unital for $q=128$ which necessarily is not one arising from the construction of Feng and Li.

\end{document}